\documentclass[12pt]{article}
\usepackage{amsmath,amssymb,amsthm,bm,amscd,color,cases}
\allowdisplaybreaks
\numberwithin{equation}{section}
\usepackage{geometry}
\theoremstyle{definition}
\newtheorem{thm}{Theorem}[subsection]
\newtheorem{defn}[thm]{Definition}
\newtheorem{lem}[thm]{Lemma}
\newtheorem{prop}[thm]{Proposition}

\newtheorem{rem}[thm]{Remark}
\newtheorem{ex}[thm]{Example}
\newtheorem{theorem}{Theorem}

\newcommand{\Fg}{\mathfrak{g}}
\newcommand{\Fh}{\mathfrak{h}}

\newcommand{\BC}{\mathbb{C}}
\newcommand{\BR}{\mathbb{R}}

\newcommand{\BZ}{\mathbb{Z}}
\newcommand{\BB}{\mathbb{B}}

\newcommand{\CB}{\mathcal{B}}

\newcommand{\Rwt}{\mathrm{wt}}
\newcommand{\Rhom}{\mathrm{Hom}}
\newcommand{\Rre}{\mathrm{re}}
\newcommand{\Rdist}{\mathrm{dist}}
\newcommand{\Rmax}{\mathrm{max}}
\newcommand{\Rmin}{\mathrm{min}}

\newcommand{\ac}{Acknowledgment}

\newcommand{\pair}{\langle \pi(t),\,\alpha_i^{\vee} \rangle}

\title{\bf{Lakshmibai-Seshadri paths for \\ 
hyperbolic Kac-Moody algebras of rank 2}}
\author{Dongxiao Yu}
\date{}
\begin{document}
\maketitle
\begin{center}
Graduate School of Pure and Applied Sciences, University of Tsukuba, 

1-1-1 Tennodai, Tsukuba, Ibaraki 305-8571, Japan

(e-mail: yudongxiao@math.tsukuba.ac.jp)

\end{center}

\begin{abstract}
Let $\Fg$ be a hyperbolic Kac-Moody algebra of rank $2$, and set $\lambda: = \Lambda_1 - \Lambda_2$, where $\Lambda_1, \Lambda_2$ are the fundamental weights for $\Fg$; note that $\lambda$ is neither dominant nor antidominant. Let $\BB(\lambda)$ be the crystal of all Lakshmibai-Seshadri paths of shape $\lambda$. We prove that (the crystal graph of) $\BB(\lambda)$ is connected. Furthermore, we give an explicit description of Lakshmibai-Seshadri paths of shape $\lambda$.
\end{abstract}

\section{Introduction.}

Let $\Fg$ be a symmetrizable Kac-Moody algebra over $\BC$ with $\Fh$ the Cartan subalgebra. We denote by $W$ the Weyl group of $\Fg$. Let $P$ be an integral weight lattice of $\Fg$, $P^+$ the set of dominant integral weights, and $-P^+$ the set of antidominant weights.
In \cite{o4, o5}, Littelmann introduced the notion of Lakshmibai-Seshadri (LS for short) paths of shape $\lambda \in P$, and gave the set $\BB(\lambda)$ of all LS paths of shape $\lambda$ a crystal structure.  
Kashiwara \cite{o12} and Joseph \cite{o10} proved independently that if $\lambda \in P^+$ (resp., $\lambda \in -P^+$), then $\BB(\lambda)$ is isomorphic, as a crystal, to the crystal basis of the integrable highest (resp., lowest) weight module of highest weight (resp., lowest weight) $\lambda$.
Since $\BB(\lambda) = \BB(w\lambda)$ for every $\lambda \in P$ and $w \in W$ by the definition of LS paths, we can easily see that if $\lambda \in P$ satisfies $W\lambda \cap P^+ \neq \emptyset$ (resp., $W\lambda \cap (-P^+) \neq \emptyset$), then $\BB(\lambda)$ is isomorphic, as a crystal, to the crystal basis of an integrable highest (resp., lowest) module.
Here are natural questions: How are the crystal structure of $\BB(\lambda)$ and its relation to the representation theory in the case that $\lambda \in P$ satisfies $W\lambda \cap (P^+ \cup (-P^+)) = \emptyset$ ? 

If $\Fg$ is of finite type, then there is no $\lambda \in P$ such that $W\lambda \cap (P^+ \cup (-P^+)) = \emptyset$; it is well-known that $W\lambda \cap P^+ \neq \emptyset$ for any $\lambda \in P$.
Assume that $\Fg$ is of affine type, and let $c \in \Fh$ be the canonical central element of $\Fg$. Then, $W\lambda \cap (P^+ \cup (-P^+)) = \emptyset$ if and only if $(\lambda \neq 0$, and) $\langle \lambda, c\rangle = 0$. Naito and Sagaki proved in \cite{o8} and \cite{o1} that if $\lambda$ is of the form: $\lambda = m\varpi_i$, where $m \in \BZ_{\geq1}$ and $\varpi_i$ is the level-zero fundamental weight (note that $\langle \varpi_i, c\rangle = 0$), then $\BB(m\varpi_i)$ is isomorphic, as a crystal, to the crystal basis $\CB(m\varpi_i)$ of the extremal weight module of extremal weight $m\varpi_i$ over the quantum affine algebra $U_q(\Fg)$. 
Here, (for an arbitrary symmetrizable Kac-Moody algebra $\Fg$ and an arbitrary integral weight $\lambda \in P$ for $\Fg$,) the extremal weight module of extremal weight $\lambda$ is the integrable $U_q(\Fg)$-module generated by a single element $v_\lambda$ with the defining relation that ``$v_\lambda$ is an extremal weight vector of weight $\lambda$"; this module was introduced by Kashiwara \cite[\S8]{o11} as a natural generalization of integrable highest (or lowest) weight modules, and has a crystal basis $\CB(\lambda)$ (\cite[\S8]{o11}). Then, in \cite{o15}, they determined the crystal structure of $\BB(\lambda)$ for general $\lambda \in P$ such that $\langle \lambda, c\rangle = 0$, and in \cite{o16}, they proved that there exists a canonical surjective (not bijective in general) strict morphism of crystals from $\CB(\lambda)$ onto $\BB(\lambda)$.

So, in this paper, we consider the case where $\Fg = \Fg(A)$ is a hyperbolic Kac-Moody algebra of rank $2$, associated to the generalized Cartan matrix
$$A = \begin{pmatrix}
 2&-a \\
-b&2
\end{pmatrix} \ (a, b \in \BZ_{\geq 0}, ab >4).$$
Let $\Lambda_1, \Lambda_2$ be the fundamental weights for $\Fg$, and set $\lambda: = \Lambda_1 - \Lambda_2$. In Proposition \ref{n24}, we prove that $W\lambda \cap (P^+ \cup (-P^+)) = \emptyset$ if $a, b \geq 2$; in fact, if $a = 1$ (resp., $b = 1$), then $W\lambda \cap P^+ \neq \emptyset$ (resp., $W\lambda \cap (-P^+) \neq \emptyset$); see Remark \ref{n25}. Then we prove the following theorem.

\begin{theorem}
The crystal graph of $\BB(\Lambda_1 - \Lambda_2)$ is connected.
\end{theorem}

Our weight $\lambda = \Lambda_1 - \Lambda_2$ can be considered as an analog of the level-zero fundamental weight for a (rank 2) affine Lie algebra. So, in a future work, we will study, as in the affine case, the relation between $\BB(\Lambda_1 - \Lambda_2)$ and the crystal basis $\CB(\Lambda_1 - \Lambda_2)$ of the extremal weight module of extremal weight $\Lambda_1 - \Lambda_2$. In (the proof of $\BB(\lambda) \cong \CB(\lambda)$ for $\lambda \in P^+$ in \cite{o12} and \cite{o10}, and) the proof of $\BB(\varpi_i) \cong \CB(\varpi_i)$ in \cite{o8} and \cite{o1}, the connectedness of these crystals played very important roles. Therefore, Theorem 1 will be strongly related to the representation theory. 

Finally, in the case where $a, b \geq 2$, we give an explicit description of LS paths of shape $\Lambda_1- \Lambda_2$.

\begin{theorem}
Assume that $a, b \geq 2$. An LS path of shape $\lambda = \Lambda_1 - \Lambda_2$ is of the form (i) or (ii):

\noindent(i) $(x_{m+s-1}\lambda, \dots, x_{m+1}\lambda, x_m\lambda ; \sigma_0, \sigma_1, \dots, \sigma_s)$, where $m \geq 0, s \geq 1$, and $0 = \sigma_0 < \sigma_1 < \dots < \sigma_s = 1$ satisfy the condition that $p_{m+s-u} \sigma_u \in \BZ$ for $1 \leq u \leq s-1$.

\noindent(ii) $(y_{m-s+1}\lambda, \dots, y_{m-1}\lambda, y_m\lambda ; \delta_0, \delta_1, \dots, \delta_s)$, where $m \geq s-1, s \geq 1$, and $0 = \delta_0 < \delta_1 < \dots < \delta_s = 1$ satisfy the condition that $q_{m-s+u+1} \delta_u \in \BZ$ for $1 \leq u \leq s-1$.

\noindent Here, the elements $x_m, y_m \in W, m \geq 0,$ are defined in (\ref{eq:eq01}), (\ref{eq:eq02}), and the sequences $\{p_m\}_{m \geq 0}$ and $\{q_m\}_{m \geq 0}$ are defined in (\ref{eq:eq05}), (\ref{eq:eq06}).
\end{theorem}

This paper is organized as follows. In Section $2$, we fix our notation, and recall the definitions and several properties of LS paths. In Section $3$, after showing some lemmas, we give a proof of Theorem 1. In Section $4$, we give the explicit description of LS paths of shape $\lambda = \Lambda_1 - \Lambda_2$ (Theorem 2), after showing some technical lemmas.

\section{Preliminaries.}
\subsection{Hyperbolic Kac-Moody algebra of rank 2.}

Let 
\begin{equation}
\begin{split}
\label{eq:eq07}
A = \begin{pmatrix} 
   2&-a \\
 -b&2
\end{pmatrix},  \ \text{where} \ a,b \in \BZ_{>0} \ \text{and} \ ab > 4,
\end{split}
\end{equation}
be a hyperbolic generalized Cartan matrix of rank $2$. Let $\Fg = \Fg(A)$ be the Kac-Moody algebra  associated to $A$ over $\BC$. Denote by $\Fh$ the Cartan subalgebra of $\Fg$, $\{\alpha_1, \alpha_2\} \subset \Fh^{\ast} : = \Rhom_{\BC}(\Fh, \BC)$ the set of simple roots,  and $\{\alpha_1^{\vee}, \alpha_2^{\vee}\} \subset \Fh$ the set of simple coroots; we set $I = \{1, 2\}$. 
We denote by $W = \langle r_1, \ r_2 \rangle$ the Weyl group of $\Fg$, where $r_i$ is the simple reflection in $\alpha_i$ for $i = 1, 2$; note that $W = \{x_m, y_m \mid m \in \BZ_{\geq 0}\}$, where
\begin{equation}
x_m := 
\label{eq:eq01}
\begin{cases}
(r_2r_1)^k  &\text{if $m = 2k$ with $k \in \BZ_{\geq 0}$},\\
r_1(r_2r_1)^k  &\text{if $m = 2k + 1$ with $k \in \BZ_{\geq 0}$}.
\end{cases} 
\end{equation}
\begin{equation}
y_m := 
\label{eq:eq02}
\begin{cases}
(r_1r_2)^k  &\text{if $m = 2k$ with $k \in \BZ_{\geq 0}$},\\
r_2(r_1r_2)^k  &\text{if $m = 2k + 1$ with $k \in \BZ_{\geq 0}$}.
\end{cases}
\end{equation}
Let $\Delta_{\Rre}^+$ denote the set of positive real roots. 
We see that 
\begin{equation}
\begin{split}
\label{eq:eq08}
\Delta_{\Rre}^+ = \big\{x_l(\alpha_2), \ y_{l+1}(\alpha_1) \mid l \in \BZ_{\text{even} \geq 0} \big\} \sqcup 
\big\{y_l(\alpha_1),  x_{l+1}(\alpha_2) \mid l \in \BZ_{\text{even} \geq 0} \big\},
\end{split}
\end{equation}
where $\BZ_{\text{even} \geq 0}$ denotes the set of even nonnegative integers.
\begin{rem}
In fact, we know from \cite[5.25]{o6} that
$$\Delta_{\Rre}^+ = \left\{c_j\alpha_1 + d_{j+1}\alpha_2 \ \text{and} \ c_{j+1}\alpha_1 + d_j\alpha_2 \mid j \geq 0 \right\},$$
where the sequences $\{c_j\}_{j \geq 0}$ and $\{d_j\}_{j \geq 0}$ are defined by 
\begin{equation*}
\begin{split}
c_0 = d_0 = 0, \ d_1 = c_1 = 1, \ \text{and} \
\begin{cases}
c_{j+2} = a d_{j+1} - c_j, \\
d_{j+2} = b c_{j+1} - d_j.
\end{cases}
\end{split}
\end{equation*}
\end{rem}

\noindent For a positive real root $\beta \in \Delta_{\Rre}^+$, we denote by $\beta^{\vee}$ the dual root of $\beta$, and by $r_{\beta} \in W$ the reflection in $\beta$.



Let $\Lambda_1, \Lambda_2 \in \Fh^{\ast}$ be the fundamental weights for $\Fg$, i.e., $\langle \Lambda_i,  \alpha_j^{\vee}\rangle = \delta_{i, j}$ for $i, j = 1, 2$, and set $P := \BZ\Lambda_1 \oplus \BZ\Lambda_2$.
Let $P^+ := \BZ_{\geq 0}\Lambda_1 + \BZ_{\geq 0}\Lambda_2 \subset P$ be the set of dominant integral weights,  and $-P^+$ the set of antidominant integral weights.

\subsection{Lakshmibai-Seshadri paths.}

Let us recall the definition of Lakshmibai-Seshadri paths from \cite[\S4]{o5}.

\begin{defn}\label{a4}
Let $\lambda \in P$ be an integral weight. For $\mu, \nu \in W\lambda$, we write $\mu \geq \nu$ if there exist a sequence $\mu = \mu_0, \ \mu_1, \ \dots, \ \mu_r = \nu$ of elements in $W\lambda$ and a sequence $\beta_1, \ \beta_2, \ \dots, \ \beta_r$ of positive real roots such that $\mu_k = r_{\beta_k}(\mu_{k-1})$ and $\langle \mu_{k-1}, \beta_k^{\vee}\rangle < 0$ for $k = 1, 2, \ \dots, \ r$. The sequence $\mu_0, \ \mu_1, \ \dots, \ \mu_r$ above is called a chain for $(\mu, \nu)$. If $\mu \geq \nu$, then we define $\mathrm{dist}(\mu, \nu)$ to be the maximal length $r$ of all possible chains for $(\mu, \nu)$. 
\end{defn}

\begin{rem}\label{a9}
Let $\mu, \nu \in W\lambda$ be such that $\mu > \nu$ with $\mathrm{dist}(\mu, \nu) = 1$. Then there exists a unique $\beta \in \Delta_{\text{re}}^+$ such that $r_\beta(\mu) = \nu$.
\end{rem}

Let $\lambda \in P$. The Hasse diagram of $W\lambda$ is, by definition, the $\Delta_{\Rre}^+$-labeled, directed graph with vertex set $W\lambda$, and edges of the following form: $\nu \xrightarrow{\beta} \mu$ for $\mu, \nu \in W\lambda$ and $\beta \in \Delta_{\Rre}^+$ such that $\mu > \nu$ with $\Rdist(\mu, \nu) = 1$ and $\nu = r_{\beta}(\mu)$.



\begin{defn}\label{a5}
Let $\lambda \in P, \ \mu, \nu \in W\lambda$ with $\mu >\nu$, and $0 < \sigma < 1$ a rational number. A $\sigma$-chain for $(\mu, \nu)$ is,  by definition, a decreasing sequence $\mu = \mu_0 \ > \mu_1 \ > \ \cdots \ > \mu_r = \nu$ of elements in $W\lambda$ such that $\mathrm{dist}(\mu_{k - 1}, \mu_k) = 1$ and $\sigma\langle \mu_{k-1}, \beta_k^{\vee} \rangle \in \BZ_{< 0}$ for all $k = 1, \ 2, \ \dots, \ r$, where $\beta_k$ is the unique positive real root such that $\mu_k = r_{\beta_k}(\mu_{k-1})$. 
\end{defn}

\begin{defn}\label{a6}
Let $\lambda \in P$. Let $(\underline{\nu}; \ \underline{\sigma})$ be a pair of a sequence $\underline{\nu}$: $\nu_1 > \nu_2 > \dots > \nu_s$ of elements in $W\lambda$ and a sequence $\underline{\sigma}: 0 = \sigma_0 < \sigma_1 < \dots < \sigma_s = 1$ of rational numbers, where $s \geq 1$. The pair $(\underline{\nu}; \underline{\sigma})$ is called a  Lakshmibai-Seshadri (LS for short) path of shape $\lambda$,  if for each $k = 1, 2, \dots, s-1$, there exists a $\sigma_k$-chain for $(\nu_k, \nu_{k + 1})$. We denote by $\BB(\lambda)$ the set of all LS paths of shape $\lambda$.
\end{defn}

We identify $\pi = (\nu_1, \ \nu_2, \ \dots, \ \nu_s ; \ \sigma_0, \ \sigma_1, \ \dots, \ \sigma_s) \in \BB(\lambda)$ with the following piecewise-linear continuous map $\pi : [0, 1] \rightarrow \BR \otimes_{\BZ} P :$
$$\pi(t) = \sum_{k = 1}^{j - 1} (\sigma_k - \sigma_{k - 1})\nu_k + (t - \sigma_{j-1})\nu_j \ \text{for} \ \sigma_{j - 1} \leq t \leq \sigma_j, \ 1 \leq j \leq s.$$

Now, we endow $\BB(\lambda)$ with a crystal structure as follows ; for the axiom of  crystal, see \cite[Definition 4.5.1]{o9}.
First, we define $\Rwt(\pi) := \pi(1)$ for $\pi \in \BB(\lambda)$; we know from \cite[Lemma 4.5]{o5} that $\pi(1) \in P$. Next, for $\pi \in \BB(\lambda)$ and $i \in I$,  we define
\begin{equation}\label{eq:eq105}
\begin{split}
H_i^{\pi}(t) := \pair \ \text{for} \ t \in [0, 1], \ \ m_i^{\pi} := \mathrm{min}\{H_i^{\pi}(t) \mid t \in [0, 1]\}.
\end{split}
\end{equation}
We know from \cite[\S2.6]{o5} that 
\begin{equation}\label{eq:eq106}
\text{all local minimal values of }H_i^{\pi}(t) \text{ are integers;}
\end{equation}
in particular, $m_i^{\pi}$ is a nonpositive integer, and $H_i^{\pi}(1) - m_i^{\pi}$ is a nonnegative integer. 
We define $e_i\pi$ as follows: If $m_i^{\pi} = 0$, then we set $e_i \pi := \mathbf{0}$, where $\mathbf{0}$ is an extra element not contained in any crystal. If $m_i^{\pi} \leq -1$, then we set
\begin{equation}\label{eq:eq107}
\begin{split}
t_1 &:= \mathrm{min}\{t \in [0, 1] \mid H_i^{\pi}(t) = m_i^{\pi}\}, \\
t_0 &:= \mathrm{max}\{t \in [0, t_1] \mid H_i^{\pi}(t) = m_i^{\pi} + 1\};
\end{split}
\end{equation}
by (2.6), we see that
\begin{equation}\label{eq:eq108}
H_i^{\pi}(t) \ \text{is strictly decreasing on} \ [t_0, t_1].
\end{equation}
We define
\begin{equation*}
\begin{split}
(e_i \pi)(t) := \begin{cases}
\pi(t)  &\text{if $0 \leq t \leq t_0$},\\
r_i(\pi(t) - \pi(t_0)) + \pi(t_0)  &\text{if $t_0 \leq t \leq t_1$},\\
\pi(t) + \alpha_i  &\text{if $t_1 \leq t \leq 1$};
\end{cases}
\end{split}
\end{equation*}
we know from \cite[\S4]{o5} that $e_i\pi \in \BB(\lambda)$.

Similarly, we define $f_i \pi$ as follows :  If $H_i^{\pi}(1) - m_i^{\pi} = 0$, then we set $f_i \pi := \mathbf{0}$. If $H_i^{\pi}(1) - m_i^{\pi} \geq 1$, then we set
\begin{equation}\label{eq:eq109}
\begin{split}
t_0 &:= \mathrm{max}\{t \in [0, 1] \mid H_i^{\pi}(t) = m_i^{\pi}\},\\
t_1 &:= \mathrm{min}\{t \in [t_0, 1] \mid H_i^{\pi}(t) = m_i^{\pi} + 1 \};
\end{split}
\end{equation}
by (2.6), we see that 
\begin{equation}\label{eq:eq110}
H_i^{\pi}(t) \ \text{is strictly increasing on} \ [t_0, t_1].
\end{equation}
We define
 \begin{equation*}
\begin{split}
(f_i \pi)(t) := \begin{cases}
\pi(t)  &\text{if $0 \leq t \leq t_0$},\\
r_i(\pi(t) - \pi(t_0)) + \pi(t_0)  &\text{if $t_0 \leq t \leq t_1$},\\
\pi(t) - \alpha_i  &\text{if $t_1 \leq t \leq 1$};
\end{cases}
\end{split}
\end{equation*}
we know from \cite[\S4]{o5} that $f_i\pi \in \BB(\lambda)$. We set $e_i\mathbf{0} = f_i\mathbf{0} := \mathbf{0}$ for $i \in I$.

Finally, for $\pi \in \BB(\lambda)$ and $i \in I$, we set 
\begin{equation*}
\begin{split}
\varepsilon_i(\pi ) &:= \mathrm{max}\{n \in \BZ_{\geq 0} \mid e_i^{n}\pi \neq \mathbf{0}\},\\
\varphi_i(\pi ) &:= \mathrm{max}\{n \in \BZ_{\geq 0} \mid f_i^{n}\pi \neq \mathbf{0}\}.
\end{split}
\end{equation*}

\begin{thm}[{\cite[\S2 and \S4]{o5}}]\label{n8}
The set $\BB(\lambda)$, together with the maps $\Rwt : \BB(\lambda) \rightarrow P, \ e_i, f_i : \BB(\lambda) \cup \{\mathbf{0}\} \rightarrow \BB(\lambda) \cup \{\mathbf{0}\}, i \in I,$ and $\varepsilon_i, \varphi_i : \BB(\lambda) \rightarrow \BZ_{\geq 0}, i \in I$,
becomes a crystal.
\end{thm}

For $\pi = (\nu_1, \ \nu_2, \ \dots, \ \nu_s ; \ \sigma_0, \ \sigma_1, \ \dots, \ \sigma_s) \in \BB(\lambda)$, we set $\iota(\pi) := \nu_1 \ \text{and} \ \kappa(\pi) := \nu_s.$
For $\pi \in \BB(\lambda)$ and $i \in I$,  we set $f_i^{\mathrm{max}}\pi := f_i^{\varphi_i(\pi)}\pi$ and  $e_i^{\mathrm{max}}\pi := e_i^{\varepsilon_i(\pi)}\pi$.

\begin{lem}[{\cite[Proposition 4.2]{o4}}, {\cite[Proposition 4.7]{o5}}]\label{n23}
Let $\pi \in \BB(\lambda)$, and $i \in I$. If $\langle \kappa(\pi), \ \alpha_i^{\vee}\rangle > 0$, then $\kappa(f_i^{\mathrm{max}}(\pi)) = r_i\kappa(\pi)$. If $\langle \iota(\pi), \ \alpha_i^{\vee}\rangle < 0$, then $\iota(e_i^{\mathrm{max}}(\pi)) = r_i\iota(\pi)$.
\end{lem}

\section{Connectedness of the crystal $\BB(\Lambda_1 - \Lambda_2)$.}
\subsection{An integral weight whose Weyl group orbit does not intersect with neither $P^+$ nor $-P^+$.}

If $\lambda \in P^+$ (resp., $\lambda \in -P^+$), then $\BB(\lambda)$ is isomorphic, as a crystal, to the crystal basis of the integrable highest (resp., lowest) weight module of highest (resp., lowest) weight $\lambda$ (see Kashiwara \cite{o11} and Joseph \cite{o10}). Also, we see by the definition of LS paths that $\BB(w\lambda) = \BB(\lambda)$ for every $\lambda \in P$ and $w \in W$. Hence, if $\lambda \in P$ satisfies $W\lambda \cap P^+ \neq \emptyset \ (\text{resp.,}\ W\lambda \cap (-P^+) \neq \emptyset)$, then $\BB(\lambda)$ is isomorphic to the crystal basis of an integrable highest (resp., lowest) weight module. So,  we focus on the case that $\lambda \in P$ satisfies $W\lambda \cap (P^+ \cup (-P^+)) = \emptyset$. The following proposition gives a ``fundamental'' example of such $\lambda$.

\begin{prop}\label{n24}
Assume that $a, b \neq 1$ in (\ref{eq:eq07}). If $\lambda = \Lambda_1- \Lambda_2$, then $W\lambda \cap (P^+ \cup (-P^+)) = \emptyset$.
\end{prop}

\begin{rem}\label{n25}
If $a=1$, then we have $y_1(\Lambda_1- \Lambda_2)= r_2(\Lambda_1- \Lambda_2) = \Lambda_2 \in P^+$. If $b = 1$, then we have $x_1(\Lambda_1- \Lambda_2)= r_1(\Lambda_1- \Lambda_2) = -\Lambda_1 \in -P^+$.
\end{rem}

Keep the setting in Proposition \ref{n24}. We define $\{p_m\}_{m \in \BZ_{\geq 0}}$ and $\{q_m\}_{m \in \BZ_{\geq 0}}$ by:
\begin{equation}
\begin{split}
p_0 = p_1 = 1 \ \text{and} \ \label{eq:eq05}p_{m+2} = \begin{cases}
b p_{m+1} - p_m &(\text{if $m$ is even}),\\
a p_{m+1} -p_m &(\text{if $m$ is odd}),
\end{cases}
\end{split}
\end{equation}
\begin{equation}
\begin{split}
q_0 = q_1 = 1 \ \text{and} \ \label{eq:eq06}q_{m+2} = \begin{cases}
a q_{m+1} - q_m &(\text{if $m$ is even}),\\
b q_{m+1} -q_m &(\text{if $m$ is odd}).
\end{cases}
\end{split}
\end{equation}
Then we see that $1 = p_0 = p_1 \leq p_2 < p_3 < \cdots$ and $1 = q_0 = q_1 \leq q_2 < q_3 < \cdots$; note that $p_1 = p_2$ if and only if $b = 2$, and $q_1 = q_2$ if and only if $a = 2$. Proposition \ref{n24} follows immediately from the following lemmas and the fact that $W = \{x_m, y_m \mid m \in \BZ_{\geq 0}\}$ (see (\ref{eq:eq01}) and (\ref{eq:eq02})).

\begin{lem}\label{a10}
Keep the setting in Proposition \ref{n24}. For $m \in \BZ_{\geq 0}$,
\begin{equation}
\begin{split}
\label{eq:eq03}
x_m\lambda =
\begin{cases}
p_{m+1}\Lambda_1 - p_{m}\Lambda_2 &\text{if $m$ is even},\\
-p_{m}\Lambda_1 + p_{m+1}\Lambda_2 &\text{if $m$ is odd},
\end{cases}
\end{split}
\end{equation}
\begin{equation}
\begin{split}
\label{eq:eq04}
y_m\lambda =
\begin{cases}
q_{m}\Lambda_1 - q_{m+1}\Lambda_2 &\text{if $m$ is even},\\
-q_{m+1}\Lambda_1 + q_{m}\Lambda_2 &\text{if $m$ is odd}.
\end{cases}
\end{split}
\end{equation}
\end{lem}

\begin{proof}
We give a proof only for (\ref{eq:eq03}); the proof for (\ref{eq:eq04}) is similar. We show (\ref{eq:eq03}) by induction on $m$. If $m= 0$ or $m= 1$, then (\ref{eq:eq03}) is obvious. Assume that $m > 1$. If $m$ is even, then
\begin{equation*}
\begin{split}
x_{m+1}\lambda &= r_1(x_m\lambda)
= r_1(p_{m+1}\Lambda_1 - p_m\Lambda_2)
= -p_{m+1}\Lambda_1 + (b p_{m+1} - p_m)\Lambda_2.
\end{split}
\end{equation*}
Since $m$ is even, we have $b p_{m+1} - p_{m} = p_{m+2}$ by the definition (\ref{eq:eq05}). Therefore, we obtain
$x_{m+1}\lambda = -p_{m+1}\Lambda_1 + p_{m+2}\Lambda_2$, as desired.
If $m$ is odd, then
\begin{equation*}
\begin{split}
x_{m+1}\lambda = r_2(x_m \lambda)
= r_2(-p_{m}\Lambda_1 + p_{m+1}\Lambda_2)
= (a p_{m+1} - p_m)\Lambda_1 - p_{m+1}\Lambda_2.
\end{split}
\end{equation*}
Since $m$ is odd, we have $a p_{m+1} - p_{m} = p_{m+2}$ by the definition (\ref{eq:eq06}). Therefore, we obtain
$x_{m+1}\lambda= p_{m+2}\Lambda_1 - p_{m+1}\Lambda_2$, as desired.
\end{proof}

\subsection{Connectedness.}

\begin{thm}\label{a7}
The crystal graph of $\BB(\Lambda_1 - \Lambda_2)$ is connected.
\end{thm}
If $a = 1$ or $b = 1$, then $\BB(\Lambda_1 - \Lambda_2)$ is connected by Remark \ref{n25}, together with the argument preceding  Proposition \ref{n24}. Therefore, in what follows, we assume that $a, b \neq 1$. 
In order to prove Theorem \ref{a7} in this case, we need some lemmas; we set $\lambda = \Lambda_1 - \Lambda_2$.

\begin{lem}\label{a11} 
Let $m \in \BZ_{\geq 0}$,  and $\beta \in \Delta_{\text{re}}^+$.

\noindent(1) Assume that $m$ is even. Then, $\langle x_m\lambda, \ \beta^{\vee}\rangle \in \BZ_{<0}$ if and only if $\beta = x_l(\alpha_2)$ or $y_{l+1}(\alpha_1)$ for some $l \in \BZ_{\text{even} \geq 0}$.

\noindent(2) Assume that $m$ is odd. Then, $\langle x_m\lambda, \ \beta^{\vee}\rangle \in \BZ_{<0}$ if and only if $\beta = y_l(\alpha_1)$ or $x_{l+1}(\alpha_2)$ for some $l \in \BZ_{\text{even} \geq 0}$.
\end{lem}

\begin{proof}
We give a proof only for part (1); the proof for part (2) is similar. First we show the ``if'' part of part (1). Let $l \in \BZ_{\text{even} \geq 0}$. We have $\langle x_m\lambda, x_l(\alpha_2^{\vee})\rangle = \langle x_l^{-1}x_m\lambda, \ \alpha_2^{\vee}\rangle$. Here, if $m \geq l$ (resp., $m \leq l$), then $x_l^{-1}x_m$ is equal to $x_{m-l}$ (resp., $y_{l-m}$). Therefore, by (\ref{eq:eq03}), we have $\langle x_m\lambda, x_l(\alpha_2^{\vee})\rangle = -p_{m-l} \in \BZ_{<0}$ (resp., $= -q_{l-m+1} \in \BZ_{<0}$). Similarly, we can show that $\langle x_m\lambda, y_{l+1}(\alpha_1^{\vee})\rangle < 0$.

Next, we show that the ``only if '' part of part (1); by (\ref{eq:eq08}), it suffices to show that if $\beta = x_{l+1}(\alpha_2)$ or $y_l(\alpha_1)$ for $l \in \BZ_{\text{even} \geq 0}$, then $\langle x_m\lambda, \beta^{\vee}\rangle > 0$. We have $\langle x_m\lambda, x_{l+1}(\alpha_2^{\vee}) \rangle = \langle x_{l+1}^{-1}x_m\lambda, \alpha_2^{\vee} \rangle = \langle x_{m+l+1}\lambda, \alpha_2^{\vee} \rangle$. By (\ref{eq:eq03}), we have $\langle x_m\lambda, x_{l+1}(\alpha_2^{\vee})\rangle = p_{m+l+2}>0$. Similarly, we can show that $\langle x_m\lambda, y_l(\alpha_1^{\vee})\rangle > 0$. This completes the proof of the lemma.
\end{proof}

The next lemma can be shown in exactly the same way as Lemma \ref{a11}.

\begin{lem}\label{a12}
Let $m \in \BZ_{\geq 0}$ and $\beta \in \Delta_{\text{re}}^+$.

\noindent(1) Assume that $m$ is even. Then, $\langle y_m\lambda, \ \beta^{\vee}\rangle \in \BZ_{<0}$ if and only if $\beta = x_l(\alpha_2)$ or $y_{l+1}(\alpha_1)$ for some $l \in \BZ_{\text{even} \geq 0}$.

\noindent(2) Assume that $m$ is odd. Then, $\langle y_m\lambda, \ \beta^{\vee}\rangle \in \BZ_{<0}$ if and only if $\beta = y_l(\alpha_1)$ or $x_{l+1}(\alpha_2)$ for some $l \in \BZ_{\text{even} \geq 0}$.
\end{lem}

\begin{lem}\label{a13}
(1) For $m \in \BZ_{\geq 1}$, we have $x_m\lambda > x_{m-1}\lambda$ with $\Rdist(x_m\lambda, x_{m-1}\lambda) = 1$. And $r_ix_m\lambda = x_{m-1}\lambda$, where $i = \begin{cases}
2 &(\text{if $m$ is even}),\\
1 &(\text{if $m$ is odd}).
\end{cases}$

\noindent(2) For $m \in \BZ_{\geq 1}$, we have $y_{m-1}\lambda > y_m\lambda$ with $\Rdist(y_{m-1}\lambda, y_m\lambda) = 1$. And $r_j y_m\lambda = y_{m-1}\lambda$, where $j= \begin{cases}
1 &(\text{if $m$ is even}),\\
2 &(\text{if $m$ is odd}).
\end{cases}$
\end{lem}

\begin{proof}
We give a proof only for part (1); the proof for part (2) is similar.
We see from Lemma \ref{a11} that $\langle x_m\lambda, \ \alpha_i^{\vee}\rangle < 0$. Therefore, we obtain that $x_m\lambda > r_ix_m\lambda = x_{m-1}\lambda$. Since  $\langle x_{m-1}\lambda, \ \alpha_i^{\vee}\rangle > 0$, we see by \cite[\S4 Lemma 4.1]{o5}     that $\Rdist(r_ix_m\lambda, x_{m-1}\lambda) = \Rdist(x_m\lambda, x_{m-1}\lambda) - 1$. Since $\Rdist(r_i x_m\lambda, x_{m-1}\lambda) = \Rdist(x_{m-1}\lambda, x_{m-1}\lambda) = 0$,  we obtain $\Rdist(x_m\lambda, x_{m-1}\lambda) = 1$, as desired. 
\end{proof}


\begin{prop}\label{n19}
The Hasse diagram of $W\lambda$ is 
$$\cdots \ \xleftarrow{\alpha_1} \ x_2\lambda \ \xleftarrow{\alpha_2} \ x_1\lambda \ \xleftarrow{\alpha_1} \ x_0\lambda = \lambda = y_0\lambda \ \xleftarrow{\alpha_2} \ y_1\lambda \ \xleftarrow{\alpha_1} \ y_2\lambda \ \xleftarrow{\alpha_1} \ \cdots.$$
\end{prop}

\begin{proof}
Let $\mu, \nu \in W\lambda$ be such that $\mu > \nu$ with $\Rdist(\mu, \nu) = 1$, and let $\beta \in \Delta_{\text{re}}^+$ be the (unique) positive real root such that $\nu = r_{\beta}\mu$; by Lemma \ref{a13}, it suffices to show that $\beta = \alpha_1$ or $\alpha_2$. By Lemma \ref{a11}, if $\mu = x_m\lambda$ and $m$ is even, then $\beta = x_l(\alpha_2)$ or $y_{l+1}(\alpha_1)$ for some $l \in \BZ_{\text{even} \geq 0}.$ Assume that $\beta = x_l(\alpha_2)$ for some $l \in \BZ_{\text{even} \geq 0}$; note that $r_{\beta} = (r_2r_1)^{\frac{l}{2}} r_2 (r_1r_2)^{\frac{l}{2}}$. We see from Lemma \ref{a13} that there exist a directed path 
$$\mu = x_m\lambda \ \xleftarrow{\alpha_2} \ x_{m-1}\lambda \ \xleftarrow{\alpha_1} \ \cdots \ \xleftarrow{\alpha_2} r_\beta\mu = \nu$$
of length $2l+1$ from $\mu$ to $\nu$ in the Hasse diagram of $W\lambda$. Because $\Rdist(\mu, \nu) = 1$ by assumption, we obtain $l = 0$,  and hence $\beta = \alpha_2$.
Assume that $\beta = y_{l+1}(\alpha_1)$ for some $l \in \BZ_{\text{even} \geq 0}$; note that $r_2(r_1r_2)^{\frac{l}{2}}r_1(r_2r_1)^{\frac{l}{2}}r_2$.
By the same reasoning as above,  there exists a direct path of length $2l+3 > 1$ from $\mu$ to $\nu$ in  the Hasse diagram of $W\lambda$. However,  this contradicts the assumption that $\Rdist(\mu, \nu) = 1$.
Similarly, we can show that if $\mu = x_m\lambda$ and $m$ is odd, then $\beta = \alpha_1$. Also, we can show the assertion for the case that $\mu = y_m\lambda$ in exactly the same way as above. This completes the proof of the proposition. 
\end{proof}

\begin{lem}\label{n26}
For any rational number $0 < \sigma < 1$ and any $\mu, \nu \in W\lambda$ such that $\mu > \nu$, there does not exist a $\sigma$-chain $\mu = \mu_0 > \cdots > \mu_r = \nu$ for $(\mu, \nu)$ such that $\mu_k = \lambda$ for some $0 \leq k \leq r$.
\end{lem}

\begin{proof}
Suppose that $\mu_k = \lambda$ for some $0 \leq k \leq r$. Note that $r \geq 1$ since $\mu > \nu$. If $k < r$ (resp., $k >0$), then it follows from Proposition \ref{n19} that $\mu_{k+1} = r_2\lambda$ (resp., $\mu_{k-1} = r_1\lambda$) since $\Rdist(\mu_k, \mu_{k+1}) = 1$ (resp., $\Rdist(\mu_{k-1}, \mu_k) = 1$) by the assumption of the $\sigma$-chain. Thus, we obtain $\sigma = -\sigma \langle \lambda, \alpha_2^{\vee}\rangle \in \BZ$ (resp., $\sigma = \sigma \langle \lambda, \alpha_1^{\vee}\rangle \in \BZ$), which contradicts the assumption $0 < \sigma < 1$. If $k = 0$ or $k = r$, it is clear that $\sigma = -\sigma \langle \lambda, \alpha_2^{\vee}\rangle \in \BZ$ or $\sigma = -\sigma \langle r_1\lambda, \alpha_1^{\vee}\rangle \in \BZ$ by Proposition \ref{n19}. This also contradicts the assumption. Thus, the lemma has been proved.
\end{proof}

The next proposition follows immediately from Lemma \ref{n26} and the definition of LS paths.

\begin{prop}\label{n22}
Let $\pi = (\nu_1, \ \dots, \ \nu_s \ ;\ \sigma_0, \ \dots, \ \sigma_s) \in \BB(\lambda)$. If $\nu_u = \lambda$ for some $1 \leq u \leq s$, then $s = 1$ and $\pi = (\lambda \ ; \ 0, \ 1)$.
\end{prop}

\begin{proof}[Proof of Theorem \ref{a7}]
We show that every $\pi \in \BB(\lambda)$ is connected to $(\lambda; \ 0, 1) \in \BB(\lambda)$ in the crystal graph of $\BB(\lambda)$. Assume first that $\iota(\pi) = x_m\lambda$ for some $m \in \BZ_{\geq 0}$. We show by induction on $m$ that $\pi$ is connected to $(\lambda ; \ 0, \ 1)$. If $m = 0$, then the assertion follows immediately from Proposition \ref{n22}. Assume that $m>0$. Define
$$i: = \begin{cases}
2 &(\text{if $m$ is even}),\\
1 &(\text{if $m$ is odd});
\end{cases}$$
note that $\langle x_m\lambda, \alpha_i^{\vee}\rangle < 0$ and $r_ix_m\lambda = x_{m-1}\lambda$ (see Lemma \ref{a13}). By Lemma \ref{n23}, $\iota(e_i^{\Rmax}\pi) = r_i\iota(\pi) = r_ix_m\lambda= x_{m-1}\lambda$. 
By the induction hypothesis, $e_i^{\Rmax}\pi$ is connected to $(\lambda \ ; \ 0, \ 1)$, and hence so is $\pi$.

Assume next that $\iota(\pi) = y_m\lambda$ for some $m \in \BZ_{\geq 0}$. Since $\kappa(\pi) \leq \iota(\pi)$ by the definition of an LS path, we see by Proposition \ref{n19} that $\kappa(\pi) = y_k\lambda$ for some $k \geq m$. Hence it suffices to show that if $\pi \in \BB(\lambda)$ satisfies that $\kappa(\pi) = y_k\lambda$ for some $k \in \BZ_{\geq 0}$, then $\pi$ is connected to $(\lambda; \ 0, \ 1 )$. If $k = 0$, then the assertion follows immediately from Proposition \ref{n22}. Assume that $k>0$. Define 
$$j: = \begin{cases}
1 &(\text{if $k$ is even}),\\
2 &(\text{if $k$ is odd});
\end{cases}$$ 
note that $\langle y_k\lambda, \alpha_j^{\vee}\rangle > 0$ and $r_j y_k \lambda = y_{k-1}\lambda$. By Lemma \ref{n23}, 
$\kappa(f_j^{\Rmax}\pi) = r_j\kappa(\pi) = r_jy_k\lambda= y_{k-1}\lambda$. 
By the induction hypothesis, $f_j^{\Rmax}\pi$ is connected to $(\lambda \ ; \ 0, \ 1)$, and hence so is $\pi$.
Thus, we have proved Theorem \ref{a7}. 
\end{proof}

\section{Explicit descriptions of the LS paths and the root operators.}



\subsection{Explicit description of the LS paths.}

Throughout this section, we assume that $a, b \neq 1$ in (\ref{eq:eq07}). Recall that the sequences $\{p_m\}_{m \in \BZ_{\geq 0}}$ and $\{q_m\}_{m \in \BZ_{\geq 0}}$ are defined in (\ref{eq:eq05}) and (\ref{eq:eq06}), respectively.



\begin{lem}\label{n15}
For each $k \geq 0$, the numbers $p_k$ and $p_{k+1}$ are relatively prime. Also, the numbers $q_k$ and $q_{k+1}$ are relatively prime. 
\end{lem} 

\begin{proof}
We give a proof only for $p_k$ and $p_{k+1}$; the proof for $q_k$ and $q_{k+1}$ is similar.
Suppose that the assertion is false, and let $m$ be the minimum $k \geq 0$ such that $p_k$ and $p_{k+1}$ have a common divisor greater than $1$. Let $d \in \BZ_{>1}$ be a common divisor of $p_m$ and $p_{m+1}$. Since
$$\begin{cases}
p_{m+1} = b p_m - p_{m-1} &(\text{if $m$ is even}), \\
p_{m+1} = a p_m - p_{m-1} &(\text{if $m$ is odd}),
\end{cases}$$
we can deduce that $p_m$ and $p_{m-1}$ have the same common divisor $d$, which contradicts the minimality of $m$. Thus, we have proved the lemma.
\end{proof}

\begin{thm}\label{n27}
(1) Let $0 < \sigma < 1$ be a rational number, and let $\mu, \nu \in W\lambda$ be such that $\mu > \nu$. If $\mu = \mu_0>\mu_1>\dots >\mu_t = \nu$ is a $\sigma$-chain for $(\mu, \nu)$, then $t = 1$.

\noindent(2) An LS path $\pi$ of shape $\lambda = \Lambda_1 - \Lambda_2$ is either of the form (i) or (ii):

\noindent(i) $(x_{m+s-1}\lambda, \dots, x_{m+1}\lambda, x_m\lambda ; \sigma_0, \sigma_1, \dots, \sigma_s)$, where $m \geq 0, s \geq 1$, and $0 = \sigma_0 < \sigma_1 < \dots < \sigma_s = 1$ satisfy the condition that $p_{m+s-u} \sigma_u \in \BZ$ for $1 \leq u \leq s-1$.

\noindent(ii) $(y_{m-s+1}\lambda, \dots, y_{m-1}\lambda, y_m\lambda ; \delta_0, \delta_1, \dots, \delta_s)$, where $m \geq s-1, s \geq 1$, and $0 = \delta_0 < \delta_1 < \dots < \delta_s = 1$ satisfy the condition that $q_{m-s+u+1} \delta_u \in \BZ$ for $1 \leq u \leq s-1$.
\end{thm}

\begin{proof}
(1) Suppose that $t \geq 2$. Assume first that $\mu_0 = x_m\lambda$; by Lemma \ref{n26}, we have $m \geq 3$. Since $\Rdist(\mu_0, \mu_1) = \Rdist(\mu_1, \mu_2) = 1$ by the definition of a $\sigma$-chain, we see by Proposition \ref{n19} that $\mu_1 = x_{m-1}\lambda$ and $\mu_2 = x_{m-2}\lambda$. Take $i, j \in \{1, 2\}$ such that $\mu_1 = r_i\mu_0$ and $\mu_2 = r_j\mu_1$. Then, by Lemma \ref{a10},
$$\langle \mu_0, \alpha_i^{\vee}\rangle = \langle x_m\lambda, \alpha_i^{\vee}\rangle = -p_m, \ \  \langle \mu_1, \alpha_j^{\vee}\rangle = \langle x_{m-1}\lambda, \alpha_j^{\vee}\rangle = -p_{m-1}.$$
Since $-p_m$ and $-p_{m-1}$ are relatively prime (see Lemma \ref{n15}), there does not exist a $0 < \sigma < 1$ satisfying the condition that both $-\sigma p_m$ and $-\sigma p_{m-1}$ are integers. This contradicts our assumption that $\mu = \mu_0>\mu_1>\dots >\mu_t = \nu$ is  a $\sigma$-chain for $(\mu, \nu)$. Similarly, we can get a contradiction also in the case of $\mu_0 = y_m\lambda$ for some $m \in \BZ_{\geq 1}$. Thus, we have proved (1).

(2) Let $\pi = (\nu_1, \dots, \nu_s ; \sigma_0, \dots, \sigma_s) \in \BB(\lambda)$. Assume first that $\nu_s = x_m\lambda$ for some $m \geq 0$. Since $\nu_1>\nu_2>\cdots>\nu_s = x_m\lambda$ by the definition of an LS path, we see by Proposition \ref{n19} that 
$$\nu_1 = x_{k_1}\lambda, \ \nu_2 = x_{k_2}\lambda, \ \dots, \ \nu_{s-1} = x_{k_{s-1}}\lambda$$
for some $k_1>k_2>\cdots >k_{s-1}>m$. Here we recall that there exists a $\sigma_{s-1}$-chain for $(\nu_{s-1}, \nu_s) = (x_{k_{s-1}}\lambda, x_m\lambda)$ by the definition of an LS path.
By (1), we see that the length of this $\sigma_{s-1}$-chain is equal to $1$, which implies that $\Rdist(\nu_{s-1}, \nu_s) = \Rdist(x_{k_s-1}\lambda, x_m\lambda) = 1$. Hence it follows from Proposition \ref{n19} that $k_{s-1} = m+1$. Take $i \in I$ such that $x_m\lambda = r_ix_{m+1}\lambda$. Then, by the definition of a $\sigma_{s-1}$-chain, we have $\sigma_{s-1}\langle x_{m+1}\lambda, \alpha_i^{\vee}\rangle \in \BZ$. Since $\langle x_{m+1}\lambda, \alpha_i^{\vee}\rangle = -p_{m+1}$ by (\ref{eq:eq03}), we obtain $p_{m+1}\sigma_{s-1} \in \BZ$. By repeating this argument, we deduce that $k_u = m+s-u$ and $p_{m+s-u} \sigma_u \in \BZ$ for every $1 \leq u \leq s-1$. Hence, $\pi$ is of the form (i).

Assume next that $\nu_s = y_m\lambda$ for some $m \geq 0$. Suppose that ($s \geq 2$ and) there exists $1 \leq u \leq s-1$ such that $\nu_{u+1} = y_k\lambda$ for some $k \geq 0$, but $\nu_u = x_l\lambda$ for some $l \geq 0$. By the definition of an LS path, there exists a $\delta_u$-chain for $(\nu_u, \nu_{u+1})$. Then, by (1), the length of this $\delta_u$-chain is equal to $1$, which implies that $\Rdist(\nu_u, \nu_{u+1}) = \Rdist(x_l\lambda, y_k\lambda) = 1$. By the Hasse diagram in Proposition \ref{n19}, we see that $(l, k) = (1, 0)$ or $(0, 1)$ Since $x_0\lambda = y_0\lambda = \lambda$, it follows form Proposition \ref{n22} that $s = 1$, which contradicts $s \geq 2$. Therefore, we conclude that
$$\nu_1 = y_{k_1}\lambda, \ \nu_2 = y_{k_2}\lambda, \ \dots, \ \nu_s = y_{k_s}\lambda = y_m\lambda,$$
where $0 \leq k_1 < k_2 < \cdots < k_{s-1} < k_s = m$. By the same argument as above, we deduce that $k_u = m-s+u$ and $q_{m-s+u+1}\delta_u \in \BZ$. Hence, $\pi$ is of the form (ii). This completes the proof of Theorem \ref{n27}. 
\end{proof}


\subsection{Explicit description of the root operators.}

As an application of Theorem \ref{n27}, we give an explicit description of the root operators $e_i$ and $f_i, i = 1, 2.$ First, let $\pi \in \BB(\lambda)$ be of the form (i) in Theorem \ref{n27}(2).
We set 
\begin{equation*}
\begin{split}
&C_u^{(1)} := \sum_{k=m+s-u}^{m+s-1}(\sigma_{m+s-k} - \sigma_{m+s-k-1})(-1)^{k}p_{k+\xi_k},\\
&C_u^{(2)} := \sum_{k=m+s-u}^{m+s-1}(\sigma_{m+s-k} - \sigma_{m+s-k-1})(-1)^{k+1}p_{k+\xi_{k+1}},\\
\end{split}
\end{equation*}
where $\{p_m\}_{m \in \BZ_{\geq 0}}$ is define as (\ref{eq:eq05}), and
$$\xi_k := \begin{cases}
1 &\text{if $k$ is even},\\
0 &\text{if $k$ is odd}, 
\end{cases}\ \text{for} \ k \in \BZ_{\geq 0}.$$
Then, 
\begin{equation*}
\Rwt(\pi) = C_s^{(1)} \Lambda_1 + C_s^{(2)} \Lambda_2.
\end{equation*}
Note that $\pm \langle x_u\lambda, \alpha_i^{\vee} \rangle > 0$ if and only if $\mp \langle x_u\lambda, \alpha_i^{\vee} \rangle >0$ for each $u \in \BZ_{\geq 0}$. Thus we see (cf. (\ref{eq:eq105})) that

\begin{equation*}
m_i^{\pi} = \Rmin\{C_u^{(i)} \mid 0 \leq u \leq s\}.
\end{equation*}
Let us give an explicit description of $f_i\pi$. We set

$$u_0 := \Rmax\{0 \leq u \leq s \mid C_u^{(i)} = m_i^{\pi}\};$$
if $u_0 = s$, then $f_i\pi = \mathbf{0}$. Assume that $0 \leq u_0 \leq s-1$; we see that $\sigma_{u_0}$ is equal to $t_0$ in (\ref{eq:eq109}). By fact (\ref{eq:eq110}), we deduce that $t_1$ in (\ref{eq:eq109}) is equal to 
\begin{equation*}
\begin{split}
\sigma_{u_0}' := \begin{cases}
\sigma_{u_0} + \displaystyle\frac{1}{p_{m+s-u_0-1+\xi_{m+s-u_0-1}}} &\text{if $i = 1$,} \vspace{4mm}\\
\sigma_{u_0} + \displaystyle\frac{1}{p_{m+s-u_0-1+\xi_{m+s-u_0}}} &\text{if $i = 2$, }
\end{cases}
\end{split}
\end{equation*}
which satisfies $\sigma_{u_0} < \sigma_{u_0}' \leq \sigma_{u_0+1}$; notice that if $\sigma_{u_0}' = \sigma_{u_0 + 1}$, then $u_0 = s-1$, and hence $\sigma_{u_0}' = \sigma_s = 1.$
We have 
\begin{equation}\label{100}
f_i\pi = \begin{cases}
(x_{m+s}\lambda, x_{m+s-1}\lambda, \dots,  x_m\lambda ; \sigma_0, \sigma_{0}', \sigma_1, \dots, \sigma_s) &\text{if $u_0 = 0$ and $\sigma_{u_0}' < \sigma_{u_0 + 1},$}\\
(r_ix_m\lambda; 0, 1) &\text{if $u_0 = 0$ and $\sigma_{u_0}' = \sigma_{u_0 + 1},$}\\
(x_{m+s-1}\lambda, \dots,  x_m\lambda ; \sigma_0, \dots, \sigma_{u_0-1}, \sigma_{u_0}', \sigma_{u_0 + 1}, \dots, \sigma_s) &\text{if $u_0 \geq 1$ and $\sigma_{u_0}' < \sigma_{u_0 + 1},$}\\
(x_{m+s-1}\lambda, \dots,  x_{m-1}\lambda ; \sigma_0, \dots, \sigma_{s-2}, \sigma_s) &\text{if $u_0 \geq 1$ and $\sigma_{u_0}' = \sigma_{u_0 + 1}.$}
\end{cases}
\end{equation}

Similary, we give an explicit description of $e_i\pi$ as follows. We set
$$u_1 := \Rmin\{0 \leq u \leq s \mid C_u^{(i)} = m_i^{\pi}\};$$
if $u_1 = 0$, then $e_i\pi = \mathbf{0}$. Assume that $1 \leq u_1 \leq s$; we see that $\sigma_{u_1}$ is equal to $t_1$ in (\ref{eq:eq107}). By fact (\ref{eq:eq108}), we deduce that $t_1$ in (\ref{eq:eq107}) is equal to 
\begin{equation*}
\sigma_{u_1}' := \begin{cases}
\sigma_{u_1} - \displaystyle\frac{1}{p_{m+s-u_1+\xi_{m+s-u_1}}} &\text{if $i = 1$, }\vspace{4mm}\\
\sigma_{u_1} - \displaystyle\frac{1}{p_{m+s-u_1+\xi_{m+s-u_1+1}}} &\text{if $i = 2$, }
\end{cases}
\end{equation*}
which satisfies $\sigma_{u_1-1} \leq \sigma_{u_1}' \leq \sigma_{u_1}$; notice that if $\sigma_{u_1}' = \sigma_{u_1 -1}$, then $u_1 = 1$ and hence $\sigma_{u_1}' = \sigma_0 = 0$.
We have 

\begin{equation}\label{101}
e_i\pi = \begin{cases}
(x_{m+s-1}\lambda, \dots,  x_m\lambda, x_{m-1}\lambda ; \sigma_0, \dots, \sigma_{s-1}, \sigma_{s}', \sigma_s) &\text{if $u_1 = s$ and $\sigma_{u_1 -1} < \sigma_{u_1}'$},\\
(r_ix_m\lambda; 0, 1) &\text{if $u_1 = s$ and $\sigma_{u_1-1} = \sigma_{u_1}'$},\\
(x_{m+s-1}\lambda, \dots,  x_m\lambda ; \sigma_0, \dots, \sigma_{u_1-1}, \sigma_{u_1}', \sigma_{u_1 + 1}, \dots, \sigma_s) &\text{if $u_1 \leq s-1$, $\sigma_{u_1-1} < \sigma_{u_1}' ,$}\\
(x_{m+s-2}\lambda, \dots,  x_{m}\lambda ; \sigma_0, \sigma_{2}, \dots,  \sigma_s) &\text{if $u_1 \leq s-1$, $\sigma_{u_1 - 1} = \sigma_{u_1}'.$}
\end{cases}
\end{equation}
Note that $x_{-1}\lambda = y_1\lambda.$

\begin{ex}
Let
$$\pi = (r_2r_1r_2r_1\lambda, r_1r_2r_1\lambda, r_2r_1\lambda; 0, \frac{1}{p_4}, \frac{2}{p_3}, 1) \in \BB(\lambda); $$
note that $m = 2$ and $s = 3$ in Theorem \ref{n27} (i). 
Let us compute $f_i\pi, i = 1, 2,$ using formula (\ref{100}).
If $i = 1$, then we have $u_0 = 2, \sigma_{u_0}' = \frac{3}{p_3}$ if $a = 2$ and $u_0 = 0, \sigma_{u_0}' = \frac{1}{p_5}$ if $a \geq 3$. (Note that $\frac{3}{p_3} = 1$ if $b = 3$.) Thus,
$$f_1\pi = \begin{cases}
(r_2r_1r_2r_1\lambda, r_1r_2r_1\lambda; 0, \displaystyle\frac{1}{p_4}, 1) &\text{if $a = 2, b = 3$,} \vspace{4mm}\\
(r_2r_1r_2r_1\lambda, r_1r_2r_1\lambda, r_2r_1\lambda; 0, \displaystyle\frac{1}{p_4}, \displaystyle\frac{3}{p_3}, 1) &\text{if $a = 2, b > 3$,} \vspace{4mm}\\
(r_1r_2r_1r_2r_1\lambda, r_2r_1r_2r_1\lambda, r_1r_2r_1\lambda, r_2r_1\lambda; 0, \displaystyle \frac{1}{p_5},
\displaystyle\frac{1}{p_4}, \displaystyle\frac{2}{p_3}, 1) &\text{if $a \geq 3$;}
\end{cases}$$
remark that if $a = 2$, then $b >3$. If $i = 2$, then we have $u_0 = 1, \sigma_{u_0}' = \frac{2}{p_4}$. Thus,
$$f_2\pi = (r_2r_1r_2r_1\lambda, r_1r_2r_1\lambda, r_2r_1\lambda; 0, \frac{2}{p_4}, \frac{2}{p_3}, 1).$$
\end{ex}

Next, let $\pi \in \BB(\lambda)$ be of the form (ii) in Theorem \ref{n27} (2). By a similar argument to above, we have the following explicit descriptions of $f_i\pi$ and $e_i\pi$.
We set
\begin{equation*}
\begin{split}
&D_v^{(1)} := \sum_{k=m-s+1}^{m-s+v}(\delta_{k-m+s} - \delta_{k-m+s-1})(-1)^{k}q_{k+\xi_{k+1}},\\
&D_v^{(2)} := \sum_{k=m-s+1}^{m-s+v}(\delta_{k-m+s} - \delta_{k-m+s-1}(-1)^{k+1}q_{k+\xi_{k}},
\end{split}
\end{equation*}
where $\{q_m\}_{m \in \BZ_{\geq 0}}$ is define as (\ref{eq:eq06}).
Then
\begin{equation*}
\Rwt(\pi) = D_s^{(1)} \Lambda_1 + D_s^{(2)} \Lambda_2.
\end{equation*}
We have
\begin{equation*}
m_i^{\pi} = \Rmin\{D_v^{(i)} \mid 0 \leq v \leq s\}.
\end{equation*}
Let us give an explicit description of $f_i\pi$. We set

$$v_0 := \Rmax\{0 \leq v \leq s \mid D_v^{(i)} = m_i^{\pi}\};$$
if $v_0 = s$, then $f_i\pi = \mathbf{0}$. Assume that $0 \leq v_0 \leq s-1$. we set
$$\delta_{v_0}' := \begin{cases}
\delta_{v_0} + \displaystyle\frac{1}{p_{m-s+v_0+1+\xi_{m-s+v_0+2}}} &\text{if $i = 1$, }\vspace{4mm}\\
\delta_{v_0} + \displaystyle\frac{1}{p_{m-s+v_0+1+\xi_{m-s+v_0+1}}} &\text{if $i = 2$. }
\end{cases}$$
We have
\begin{equation}\label{103}
f_i\pi = \begin{cases}
(y_{m-s}\lambda, y_{m-s+1}\lambda, \dots,  y_m\lambda ; \delta_0, \delta_{0}', \delta_1, \dots, \delta_s) &\text{if $v_0 = 0$ and $\delta_{v_0}' < \delta_{v_0 + 1},$}\\
(r_i y_m\lambda; 0, 1) &\text{if $v_0 = 0$ and $\delta_{v_0}' = \delta_{v_0 + 1},$}\\
(y_{m-s+1}\lambda, \dots,  y_m\lambda ; \delta_0, \dots, \delta_{v_0-1},\delta_{v_0}', \delta_{v_0 + 1}, \dots, \delta_s) &\text{if $v_0 \geq 1$ and $\delta_{v_0}' < \delta_{v_0 + 1},$}\\
(y_{m-s+1}\lambda, \dots,  y_{m-1}\lambda ; \delta_0, \dots, \delta_{s-2}, \delta_s) &\text{if $v_0 \geq 1$ and $\delta_{u_0}' = \delta_{v_0 + 1}.$}
\end{cases}
\end{equation}
Note that $y_{-1}\lambda = x_1\lambda.$

Similarly, we give an explicit description of $e_i\pi$ as follows. We set
$$v_1 := \Rmin\{0 \leq v \leq s \mid D_v^{(i)} = m_i^{\pi}\};$$
if $v_1 = 0$, then $e_i\pi = \mathbf{0}$. Assume that $1 \leq v_1 \leq s$. We set 
\begin{equation*}
\delta_{v_1}' := \begin{cases}
\delta_{v_1} - \displaystyle\frac{1}{q_{m-s+v_1+\xi_{m-s+v_1+1}}} &\text{if $i = 1$, }\vspace{4mm}\\
\delta_{v_1} - \displaystyle\frac{1}{q_{m-s+v_1+\xi_{m-s+v_1}}} &\text{if $i = 2$. }
\end{cases}
\end{equation*}
We have 

\begin{equation}\label{101}
e_i\pi = \begin{cases}
(y_{m-s+1}\lambda, \dots,  y_m\lambda, y_{m+1}\lambda ; \delta_0, \dots, \delta_{s-1}, \delta_{s}', \delta_s) &\text{if $v_1 = s$ and $\delta_{v_1 -1} < \delta_{v_1}'$},\\
(r_i y_m\lambda; 0, 1) &\text{if $v_1 = s$ and $\delta_{v_1-1} = \delta_{v_1}',$}\\
(y_{m-s+1}\lambda, \dots,  y_m\lambda ; \delta_0, \dots, \delta_{v_1-1}, \delta_{v_1}', \delta_{v_1 + 1}, \dots, \delta_s) &\text{if $v_1 \leq s-1$, $\delta_{v_1-1} < \delta_{v_1}' ,$}\\
(y_{m-s+2}\lambda, \dots,  y_{m}\lambda ; \delta_0, \delta_{2}, \dots,  \delta_s) &\text{if $v_1 \leq s-1$, $\delta_{v_1 - 1} = \delta_{v_1}'.$}
\end{cases}
\end{equation}

\noindent\textbf{\Large{\ac.}}

The author is grateful to Professor Daisuke Sagaki, her supervisor, for suggesting the topic treated in this paper and lending his expertise especially through the study of the connectedness of LS paths as a crystal graph. Also, she thanks the referee for giving her valuable comments.

\end{document}